\documentclass[12pt]{article}
\usepackage{amsthm,amsmath, color}
\usepackage{amsfonts}
\usepackage{graphicx}
\usepackage{latexsym,amsmath}

\newcounter{thm}[section]

\newtheorem{theor}[thm]{Theorem}
\newtheorem{cor}[thm]{Corollary}
\newtheorem{lem}[thm]{Lemma}

\def\bR{\mathbb{R}}	
\def\bE{\mathbb{E}}	
\def\bP{\mathbb{P}}	
\def\bG{\mathbb{G}}	
\def\bm{\mathbf }	

\begin{document}

\title{Modified Cox regression with current status data}
\author{Laurent Bordes\footnote{University of Pau, France; laurent.bordes@univ-pau.fr.}
\;\;\;\;\;\;
Mar\'ia Carmen Pardo\footnote{Complutense University of Madrid, Spain; mcapardo@mat.ucm.es.}\;\;\;\;\;\;\; Christian Paroissin\footnote{University of Pau, France; christian.paroissin@univ-pau.fr.}
\\ 
\\ 
Valentin Patilea\footnote{CREST-Ensai, France; patilea@ensai.fr. V. Patilea acknowledges support from the research program \emph{New Challenges for New Data} of Fondation du Risque and LCL.
}
\;\;\;\;\;\;
}
\date{\empty}
\maketitle

\begin{abstract}

In survival analysis, the lifetime  under study is not always observed. In certain applications, for some individuals, the value of the lifetime is only known to be smaller or larger than some random duration. This framework represent an extension of standard situations where the lifetime is only left or only right randomly censored. We consider the case where the independent observation units include also some covariates, and we propose  two semiparametric regression models. The new models extend the standard Cox proportional hazard model to the situation of a more complex censoring mechanism. However, like in Cox's model, in both models the nonparametric baseline hazard function still could be expressed as an explicit functional of the distribution of the observations. This allows to define the estimator of the finite-dimensional parameters as the maximum of a likelihood-type criterion which is an explicit function of the data. Given an estimate of the finite-dimensional parameter, the estimation of the baseline cumulative hazard function is straightforward.

\bigskip



\end{abstract}

\section{Introduction}

Driven by applications, there is a constant interest in time-to-event analysis to extend the predictive models to situations where the lifetimes of interest suffer from complex censoring mechanisms. Here we consider the case where instead of the lifetime of interest $T,$ one observes independent copies of a finite nonnegative duration $X$ and of a discrete variable $A\in\{0,1,2\}$   such that
\begin{equation}\label{datai}
\left\{
\begin{array}{lll}
X=T & \mbox{si} & A=0, \\
X<T & \mbox{si} & A=1, \\
X\geq T & \mbox{si} & A=2.
\end{array}
\right.
\end{equation}
Depending on the application, the inequality signs in \eqref{datai} could be  strict or not.  
Let us point out that the limit case where the event $\{A=2\}$ (resp. $\{A=1\}$) has zero probability corresponds to the usual random right-censoring (resp. left-censoring) setup, while the case where the probability of the event $\{A=0\}$ is null corresponds to the current status framework.

Let us assume that $T\in[0,\infty]$ and let $\mathbf{Z}\in\mathbb{R}^q$ be a vector of random covariates. All the random variables we consider are defined on some probability space $(\Omega,\mathcal{F},\mathbb{P}).$
Although $X$ takes values only on the real line, we allow a positive probability for the event $\{T=\infty\},$ that is we allow for \emph{cured} individuals (see, for instance Fang \emph{et al.} (2005) and Zheng \emph{et al.} (2006) and the references therein for the applications where infinity lifetimes could occur). Symmetrically, we also allow the zero lifetime to have positive probability, that is a zero-inflated law for $T$ could be taken into account (see Braekers \& Growels (2015) for some motivations).


Let $\bm{\mathcal{Z}}$ be the support of $\bm{Z}.$ The conditional probability distribution  of  $(X,A)$ given $\bm{Z}$ is characterized by
the sub-distributions functions
$$
H_k([0,t]|\bm{z})=\bP(X\leq t, A=k | \bm{Z}=\bm{z}),\quad t\geq 0, \;k\in\{0,1,2\},\; \bm{z}\in \bm{\mathcal{Z}}.
$$
Let $H_k(dt|\bm{z})$ denote the associated measures. Moreover, let 
$$
H_k([0,t])=\bP(X\leq t,A=k )
$$
be the unconditional versions of these sub-distributions. Clearly,
$$
H_k([0,t]) = \mathbb{E}(H_k([0,t]|\bm{Z})),\quad t\geq 0, \;k\in\{0,1,2\}.
$$
The conditional distribution function  of $X$   given $\bm{Z}=\bm z$ is then 
$$H([0,t]|\bm{z}) = \bP(X\leq t | \bm{Z}=\bm{z}) = H_0([0,t]|\bm{z})+H_1([0,t]|\bm{z})+H_2([0,t]|\bm{z}).$$

It is important to understand that, based on the data, one could only identify the conditional sub-distributions $H_k(\cdot|\bm{z})$. For identifying and consistently estimating the conditional law of the lifetime of interest $T$, one should introduce some assumptions on the censoring mechanism. In other words, one has to consider a latent model. Several censoring mechanisms have proposed in the case \emph{without covariates}.  Turnbull (1974) considered two censoring times $L\leq U$ such that the case $\{A=0\}$ (resp. $\{A=1\}$) (resp. $\{A=2\}$) corresponds to the event $\{L\leq T \leq U\}$ and $X=T$ (resp. $\{U\leq T\}$ and $X=U$)
(resp. $\{T\leq L\}$ and $X=L$). Patilea \& Rolin (2006b) relaxed the condition $L\leq U$ and proposed two models that could be easily illustrated using simple electric circuits with three components connected in series and/or parallel. Patilea and Rolin (2006a) extended the standard right-censoring (resp. left-censoring) model by allowing uncensored lifetimes $T$ for which one only knows that are smaller (resp. larger) than the observation $X$. This corresponds, for instance, to the case of a medical study where a disease is detected for a patient, but the onset time could not be determined from medical records, personal information, \emph{etc}, while for other patients with the disease detected the onset time is available.   
The model of Turnbull does not allow to express the law of the lifetime of interest as an explicit function of the sub-distributions $H_k$, as it is the case for the models proposed by Patilea \& Rolin (2006a, 2006b). Thus a numerical algorithm is necessary to compute Turnbull's estimator. It is important to keep in mind that \emph{any} of these latent models could be correct and useful for a specific application. The data \emph{does not allow to check} the validity of the model. 
Turnbull's model, perhaps the most popular model for data structures as we consider here, is not necessarily justified in applications where there is no natural interpretation of the variables $L$ and $U$. 

The aim of this paper is to extend the modeling of data as in equation (\ref{datai}) to the case where some covariates $\bm{Z}$ are available. Kim \emph{et al.} (2010) extended Turnbull's model to the case with covariates using a proportional hazard approach.   Here we consider the extension of the approaches proposed by  Patilea \& Rolin (2006a) imposing the same proportional hazard assumption. More precisely, we propose two novel latent models for observed lifetimes as in \eqref{datai} in the presence of covariates. Both models are well suited for data as in \eqref{datai}, and hence could be used in applications. The decision to use one of them, or the one proposed by Kim \emph{et al.} (2010), could be made \emph{only}  on the basis of additional information on the application. Current status data corresponds to $A\in \{1,2\}$. Right (resp. left) censored data corresponds to $A\in\{0,1\}$ (resp. $A\in\{0,2\}$). This explains the terminology we propose for our models: modified Cox regressions with current status lifetimes.  For each of the new models, we introduce a semiparametric estimator for the finite-dimensional parameters, together with  the corresponding baseline cumulative hazard functions estimators. Our estimators are easy to implement.

The paper is organized as follows. Our semiparametric models are introduced in section \ref{sec:right-ca}. They extend the standard right, respectively left, random censoring proportional hazard models. In section \ref{sec:semipara} we introduce the semiparametric estimators of the covariates coefficients, and the estimators of the cumulative hazard and survival functions. In particular, we provide an estimator for the cure rate and the zero-lifetime probability. The theoretical results are presented  

\section{Censored and current status lifetimes}\label{sec:right-ca}

In our models we follow the idea of Cox's semiparametric proportional hazard model. In both models we are able to express the baseline cumulative hazard function as a functional of distribution of the observations, characterized by the conditional sub-distribu\-tions $H_k(\cdot | \bm z)$ and the law of $\bm Z$, and the coefficients of the covariates. This makes that the coefficients of the covariates could be estimated by maximizing an likelihood-type criterion that is build as an explicit function of the observations. Thus the numerical aspects are very much simplified, compared to the model considered by Kim \emph{et al.} (2010). With at hand the estimate of the finite-dimensional parameters, we could easily build the estimator of the baseline cumulative hazard function. In particular, using the estimate of total mass of  the baseline cumulative hazard, we provide a simple estimate of the conditional cure rate $\mathbb P (T=\infty \mid \bm Z = \bm z)$. Similarly, we could  provide an estimator for the conditional zero-lifetime probability $\mathbb P (T= 0 \mid \bm Z = \bm z)$.
The extension to the case of mixture models, such as considered by Fang \emph{et al.} (2005), where the cure rate or the zero-lifetime could depend on possibly different set of  covariates, is left for future work.  


\subsection{Right-censoring case}\label{sec:right-c}

Let $C\in[0,\infty)$ be a random censoring time and $\Delta$ be a Bernoulli random variable with success probability $p\in(0,1]$. Let $F_T(t|\bm{z})$ and $S_T(t|\bm{z}),$ $t\in[0,\infty],$ be the conditional distribution function and survivor function of $T$ given $\mathbf{Z}=\mathbf{z}$. Similarly,  $F_C(t|\bm{z})$ and $S_C(t|\bm{z}),$ $t\in[0,\infty),$ denote  the distribution function and the survivor function of $C$.
Following Patilea \& Rolin (2006a), the latent model for $(X,A,\bm{Z})$ is  defined by:
$$
\left\{
\begin{array}{lll}
(X,A,\bm{Z})=(T,0,\bm{Z}) & \mbox{if} & 0\leq T\leq C\mbox{ and }\Delta=1, \\
(X,A,\bm{Z})=(C,1,\bm{Z}) & \mbox{if} & 0\leq C < T, \\
(X,A,\bm{Z})=(C,2,\bm{Z}) & \mbox{if} & 0\leq T\leq C\mbox{ and }\Delta=0. \\
\end{array}
\right.
$$
Let us notice that $p=1$ is the classical right-censoring limit case, while $p=0$ would correspond to the pure current status setup. The later limit case is not included in what follows since we assume $p>0.$
In the case $C<T,$ the observed outcome is not be influenced by the value of $\Delta.$

For identification purposes we consider the following assumption.

\vspace{6pt}

\noindent \textbf{A1:} Assume that:

\vspace{-10pt}

\begin{itemize}
\item[] \quad \textbf{a}) conditionally on $\bm{Z}$, the latent variables $T$ and $C$ are independent;

\vspace{-5pt}

\item[] \quad \textbf{b}) $\Delta$ and $(T,C,\bm{Z})$ are independent.
\end{itemize}
The independence assumptions allow to  write
\begin{equation}\label{model_eq1}
\left\{
\begin{array}{ccl}
H_0(dt|\bm{z})  & = & p S_C(t-|\bm{z})F_T(dt|\bm{z}), \\
H_1(dt|\bm{z})  & = & F_C(dt|\bm{z})S_T(t|\bm{z}), \\
H_2(dt|\bm{z})  & = & (1-p) F_C(dt|\bm{z})F_T(t|\bm{z}).
\end{array}
\right.
\end{equation}
The system could be solved for the quantities $p$ and $F_T(dt|\bm{z}).$  First, let us write
$$
H_0([t,\infty)|\bm{z})+pH_1([t,\infty)|\bm{z}) = p S_T(t-|\bm{z})S_C(t-|\bm{z}).
$$
Since $ S_T(t-|\bm{z})S_C(t-|\bm{z}) = H([t,\infty)|\bm{z})$, we deduce 
$$
H_0([t,\infty)|\bm{z}) = p \{H_0([t,\infty)|\bm{z})+H_2([t,\infty)|\bm{z})\}, \quad t\geq 0.
$$
Integrating out the covariate and taking $t=0$ we could derive the simple representation
\begin{equation}\label{inv_p}
p=\frac{H_0([0,\infty))}{H_0([0,\infty))+H_2([0,\infty))}=\frac{\mathbb{P}(\Delta = 1,  T\leq C)}{\mathbb{P}(T\leq C)}.
\end{equation}
Let us point out that one could replace the condition \textbf{A1b}) by the weaker condition that $\Delta$ and $(T,C)$ are independent given $\mathbf{Z}$ and still write the equations (\ref{model_eq1}) with $p$ replaced by some function of the covariates $p(\bm{Z}).$ In this case one would derive the conditional version of the representation (\ref{inv_p}), but then the estimation of $p(\bm{Z})$ would require the estimation of the conditional versions of $H_0$ and $H_2.$
For the sake of a simpler setup we suppose that $p$ does not depend on the covariates.

Next, we solve (\ref{model_eq1}) for the conditional distribution of $T.$
For this purpose we follow a proportional hazards model approach and we suppose that the risk function of $T$ given $\bm{Z}=\bm{z}$ could be written as
\begin{equation}\label{haz_rate}
\lambda(t|\bm{z})=\lambda(t)\exp(\bm{\beta}^\top\bm{z}),\quad \forall t>0, \;\forall \bm{z}\in \bm{\mathcal{Z}},
\end{equation}
where $\lambda(\cdot)$ is some unknown baseline hazard function and $\beta$ is a vector of unknown regression parameters. (Herein the vectors are matrix columns and $\beta^\top$ denotes the transposed of $\beta.$)
With this assumption, for each $\bm{z}\in \bm{\mathcal{Z}}$ and $t\geq 0,$ we could write
\begin{eqnarray*}
H_0(dt|\bm{z})&=& pF_T(dt|\bm{z})S_C(t-|\bm{z})\\
&=& \frac{F_T(dt|\bm{z})}{S_T(t-|\bm{z})} pS_T(t-|\bm{z})S_C(t-|\bm{z})\\
&=& \lambda(t)\exp(\bm{\beta}^\top\bm{z}) pS_T(t-|\bm{z})S_C(t-|\bm{z})dt\\
&=&\lambda(t)\exp(\bm{\beta}^\top\bm{z}) \left\{H_0([t,\infty)|\bm{z})+pH_1([t,\infty)|\bm{z}) \right\} dt.
\end{eqnarray*}
Hence,
$$
H_0(dt) = \mathbb{E}\{H_0(dt|\bm{Z})\} = \mathbb{E}\{\exp(\bm{\beta}^\top\bm{Z}) \left(H_0([t,\infty)|\bm{Z})+pH_1([t,\infty)|\bm{Z}) \right)\}\lambda(t)dt.
$$
Moreover,
$$
\bE\left\{\exp(\bm{\beta}^\top\bm{Z}) H_k([t,\infty)|\bm Z)\right\} =  \bE\left\{\exp(\bm{\beta}^\top\bm{Z})\mathbf{1}(X\geq t,A=k)\right\},\quad \forall t\geq 0, k=0,1.
$$
As a consequence, for any $t$ such that $\bE\left\{\exp(\bm{\beta}^\top\bm{Z}) [H_0([t,\infty)|\bm Z)+H_1([t,\infty)|\bm Z)]\right\}>0$,
\begin{equation}\label{Lam}
\lambda(t) dt =  \frac{H_0(dt) }{\bE\left\{\exp(\bm{\beta}^\top\bm{Z}) [\mathbf{1}(X\geq t,A=0) +p \mathbf{1}(X\geq t,A=1)  ]\right\}}.
\end{equation}
Thus, the baseline cumulative hazard function $\Lambda(t) = \int_{[0,t]} \lambda (s) ds$ could be expressed as a functional of the observed variables and the finite-dimensional parameters of the model~:
\begin{equation}\label{cum_haz}
\Lambda(t)  =\Lambda(t;p,\beta)  =  \int_{[0,t]} \frac{H_0(ds) }{\bE\left\{\exp(\bm{\beta}^\top\bm{Z}) [\mathbf{1}(X\geq s,A=0) +p\mathbf{1}(X\geq s,A=1) ]\right\}}.
\end{equation}
The conditional  survival function of the lifetime of interest can be expressed as
$$
S_T(t\mid \bm{z}) = \prod_{s\in(0,t]} \left(1 - \exp(\beta^\top\bm{z}) \Lambda (ds)\right).
$$
Herein, the notation $\prod_{s\in I }$ means the product-integral over the interval $I$, as formally defined in Gill \& Johansen (1990). In particular, the conditional cure probability can be expressed as
$$
 S_T(\infty\mid \bm{z}) = \prod_{s\in(0,\infty)} \left(1 - \exp( \beta^\top\bm{z}) \Lambda (ds)\right).
$$

\subsection{Left-censoring case}\label{sec:left-c}

Let $C \in ( 0, \infty ) $ be a random censoring time and $\Delta$ be a Bernoulli random variable with success probability $p\in(0,1]$. 
In this case the latent model for  $(X,A,\bm{Z})$ is defined by:
$$
\left\{
\begin{array}{lll}
(X,A,\bm{Z})=(T,0,\bm{Z}) & \mbox{if} & 0< C\leq T\mbox{ and }\Delta=1, \\
(X,A,\bm{Z})=(C,1,\bm{Z}) & \mbox{if} & 0< C \leq T, \mbox{ and }\Delta=0\\
(X,A,\bm{Z})=(C,2,\bm{Z}) & \mbox{if} & 0\leq T< C. \\
\end{array}
\right.
$$
The case $p=1$ corresponds to the classical left-censored data situation. Consider the assumptions  \textbf{A1a}) and  \textbf{A1b}). 
Then we can  write
\begin{equation}\label{model_eq1b}
\left\{
\begin{array}{ccl}
H_0(dt|\bm{z})  & = & p F_C(t|\bm{z})F_T(dt|\bm{z}), \\
H_1(dt|\bm{z})  & = & (1-p) F_C(dt|\bm{z})S_T(t- |\bm{z}), \\
H_2(dt|\bm{z})  & = & F_C(dt|\bm{z})F_T(t-|\bm{z}).
\end{array}
\right.
\end{equation}
This system also could  be solved for the quantities $p$ and $F_T(dt|\bm{z}).$  First, combining the first and the third equation, deduce
$$
H_0([0,t]|\bm{z})+pH_2([0,t]|\bm{z}) = p F_T(t|\bm{z})F_C(t|\bm{z}),
$$
so that
\begin{equation*}
p=\frac{H_0([0,\infty))}{H_0([0,\infty))+H_1([0,\infty))}.
\end{equation*}
Moreover, for each $\bm{z}\in \bm{\mathcal{Z}}$ and $t\geq 0,$ we could write
\begin{eqnarray*}
H_0(dt|\bm{z})&=& pF_T(dt|\bm{z})F_C(t|\bm{z})\\
&=& \frac{F_T(dt|\bm{z})}{F_T(t|\bm{z})} pF_T(t|\bm{z})F_C(t|\bm{z})\\
&=& R( dt | \bm{z} ) \left\{H_0([0,t]|\bm{z})+pH_2([0,t]|\bm{z}) \right\},
\end{eqnarray*}
where 
$$
R( dt | \bm{z} ) =   \frac{F_T(dt|\bm{z})}{F_T(t|\bm{z})}
$$
is the conditional reverse hazard measure. The quantity $R( dt | \bm{z} )$ could be interpreted as the conditional probability that the event occurs in the interval $[t-dt, t]$, given that the event occurs no later than $t$. 
This measure has the property 
$$
F_T(t|\bm{z}) =\prod_{s\in(t,\infty)} \left(1 - R( ds | \bm{z} ) \right)
, \qquad \forall t\geq 0.
$$
In particular,
$$
F_T(0|\bm{z}) =  \prod_{s\in(0,\infty)} \left(1 - R( ds | \bm{z} ) \right).
$$

Inspired by the  proportional hazards approach, let us consider that the conditional reverse hazard function of $T$ given $\bm{Z}=\bm{z}$ could be written as
\begin{equation}\label{haz_rate_r}
r(t|\bm{z})=r(t)\exp(\bm{\beta}^\top\bm{z}),\quad \forall t>0, \;\forall \bm{z}\in \bm{\mathcal{Z}},
\end{equation}
where $r(\cdot)$ is some unknown baseline reverse hazard function and $\beta$ is a vector of unknown regression parameters.

Similar to the right-censoring case, one can deduce
\begin{equation}\label{Lam_r}
r(t)dt  =  \frac{H_0(dt) }{\bE\left\{\exp(\bm{\beta}^\top\bm{Z}) [\mathbf{1}(X\leq t,A=0) +p \mathbf{1}(X\leq t,A=2)  ]\right\}},
\end{equation}
and the baseline cumulative reverse hazard is obtained as $R(t)=\int_{(t,\infty)} r(s)ds$.


\section{Semiparametric likelihood estimation}\label{sec:semipara}
\setcounter{equation}{0}

Let $(X_i,A_i,\bm{Z}_i),$ $1\leq i\leq n$, denote the observations that are independent copies of  $(X,A,\bm{Z})\in[0,\infty)\times \{0,1,2\}\times \bm{\mathcal{Z}}.$ In the following, we consider $\bm{\mathcal{Z}}= \mathbb{R}^{q}$  with $q$ some  positive integer. 
With observations of the covariates and of lifetimes as in \eqref{datai}, a natural likelihood-type criterion is the one considered by Kim \emph{et al.} (2010)~:
\begin{multline}\label{kim_lik}
L_n(\beta,\Lambda) = \prod_{i=1}^n \left\{\exp(\bm\beta^\top\bm Z_i)
\lambda(X_i)
\exp\left(-\exp(\bm\beta^\top\bm Z_i)\int 
\mathbf{1}(X_i > t) \Lambda(dt)\right)
\right\}^{\textbf{1}(A_i=0)}
\\
\times \left\{\exp\left(-\exp(\bm\beta^\top\bm Z_i)\int \mathbf{1}(X_i \geq  t)
\Lambda(dt)\right)\right\}^{\textbf{1}(A_i=1)}
\\
\times \left\{1-\exp\left(-\exp(\bm\beta^\top\bm Z_i)\int \mathbf{1}(X_i \geq t)
\Lambda(dt)\right)\right\}^{\textbf{1}(A_i=2)}.
\end{multline}
In this criterion, the factors involving the distribution of $(C,\bm Z^\top)^\top$ are dropped, as they are supposed uninformative.  

To write the likelihood-type criterion $L_n(\beta,\Lambda) $, we only used a hazard rate as in \eqref{haz_rate}, without specifying any censoring mechanism or latent model. Alternatively, one could write the likelihood in terms of the cumulative reverse hazard $R(\cdot)$ we defined in section \ref{sec:left-c}, using only the assumption \eqref{haz_rate_r}. The two criteria are equivalent and would be valid for the type of data we consider. Next, one could follow the profiling idea. In the case where $\mathbb P (A=2) = 0$ this leads to Cox's  partial likelihood with right-censored data. See Murphy \& van der Vaart (2000).
A similar situation, Cox's partial likelihood with left-censored data, occurs when  $\mathbb P (A=1) = 0$. Unfortunately, given a  value $\beta$, the maximization with respect to $\Lambda (\cdot)$ (or $R(\cdot)$) of $L_n(\beta,\Lambda)$ does not have a nondegenerate, explicit solution when both $\mathbb P (A=1)$ and $ \mathbb P (A=2) $ are positive. See  Kim \emph{et al.} (2010), the Remark on page 1341. A possible solution, proposed by Kim \emph{et al.}, would be to consider a numerical approximation. Here we propose an alternative, more convenient and sound route. To estimate the parameters of interest, one has  to consider a model for the censoring mechanism. In the model considered by Kim \emph{et al.} (2010), there is no way to connect the infinite-dimensional parameter $\Lambda$ (or $R(\cdot)$) to the quantities that could be easily estimated from the data, such as $H_0(\cdot)$. This makes the profiling approach  complicated. The profiling approach is very appealing in the standard right-censoring (resp. left-censoring) case because there $\Lambda$ could be easily expressed in terms of $H_0(\cdot)$, $H_0(\cdot \mid \bm  z)$ and $H_1(\cdot \mid \bm  z)$ (resp. $H_2(\cdot \mid \bm  z)$).

In the two models we propose, the relationship between quantities that could be estimated by sample means from the data  and the infinite-dimensional parameter $\Lambda$ (or $R(\cdot)$) is explicit and this allows us to build a user-friendly approximated likelihood. These models does not only make the optimization of the likelihood-type criteria simpler. First of all, they induce censoring mechanisms that make sense in some applications.  See Patilea \& Rolin (2006a) for a discussion.


\subsection{The right-censoring and current status data case}\label{sec:profile_right}
The parameters of our first model are  $\theta=(p,\beta^\top)^\top\in(0,1]\times B \subset \mathbb{R}^{q+1}$ and the hazard function $\Lambda(\cdot)$. Let $\theta_0=(p_0,\beta_0^\top)^\top$ and $\Lambda_0(\cdot) $ denote the true values of the parameters. Using the notation from equation (\ref{cum_haz}) we can also write $\Lambda_0 (t) = \Lambda(t;\theta_0)$.

In view of equation (\ref{inv_p}) let us consider
$$
\widehat p=\frac{\sum_{i=1}^n \textbf{1}(A_i=0)}{\sum_{i=1}^n \textbf{1} (A_i\neq 1)}
$$
as estimator of $p_0.$
For estimating $\beta_0$ we shall use a partial likelihood approach. With at hand an estimate of $\beta_0$, we will use an empirical version of equation (\ref{cum_haz}) and build an estimate of $\Lambda_0(\cdot)$. For these purposes let us define empirical quantities
$$
N_{ki}(t)=\mathbf{1}(X_i\leq t,A_i=k), \qquad 1\leq i\leq n, \;k\in\{0,2\},
$$ 
$$
N_{n,0}(t)=\frac{1}{n}\sum_{i=1}^n\mathbf{1}(X_i\leq t,A_i=0)=\frac{1}{n}\sum_{i=1}^n N_{0i}(t).
$$
For a column vector $c$, $c ^{\otimes 0} =1,$ $c ^{\otimes 1} =c$ and $c ^{\otimes 2} =cc^\top.$
Let
$$
S_{n,k}^{(l)}(t;\bm{\beta})=\frac{1}{n}\sum_{i=1}^n \exp(\bm{\beta}^\top\bm{Z}_i)\bm{Z}_i^{\otimes l}\mathbf{1}(X_i\geq t,A_i=k),\qquad l=0,1,\quad k\in\{0,1,2\},
$$
and 
\begin{equation}\label{eqE}
E^{(l)}_n(t;\theta) = E^{(l)}_n(t;p, \beta)=S_{n,0}^{(l)} (t;\bm \beta) + p S_{n,1}^{(l)} (t;\bm \beta).
\end{equation}
Consider
$$
\Lambda_n(t;\theta) = \Lambda_n(t;p,\bm\beta)=\int_{[0,t]} \frac{N_{n,0}(ds)}{S_{n,0}^{(0)}(s;\bm \beta)+pS_{n,1}^{(0)}(s;\bm \beta)} = \int_{[0,t]} \frac{N_{n,0}(ds)}{E^{(0)}_n(s;\theta)}
$$
as the empirical version of the cumulative hazard function $\Lambda(t)$, as defined in \eqref{cum_haz}.


Using these empirical quantities, and recalling that $\mathbb{P}(T=C)=0$,  we can write the following approximation of the criterion defined in \eqref{kim_lik}~:
\begin{multline*}
\prod_{i=1}^n \left\{\prod_{t\in [0,\tau]} \left[\exp(\bm\beta^\top\bm Z_i)
\Lambda_n(t;\theta)\right]^{N_{0i}(dt)}
\left[\!1-\exp\left(-\int_{[0,X_i]}\exp(\bm\beta^\top\bm Z_i)
\Lambda_n(ds;\theta)\right)\right]^{N_{2i}(dt)}\right\}\\
\times\exp\left(\!-\int_{[0,\tau]} 
\{S^{(0)}_0(t;\beta)+S^{(0)}_1(t;\beta)\}\Lambda_n(dt;\theta)\right),
\end{multline*}
where $\tau\in(0,\infty)$ is some threshold that prevents from dividing by zero, it will be specified below.
Hence, let us define the approximate log-likelihood function
\begin{eqnarray*}
\ell_n(p,\bm \beta;\tau)
&=&\frac{1}{n}\sum_{i=1}^n  D_{0i}^\tau \left(\bm\beta^\top\bm Z_i-\log\left(E^{(0)}(X_i;p,\bm \beta)\right)\right)\\
&+& \frac{1}{n}\sum_{i=1}^n D_{2i}^\tau \log\left(1-\exp\left( -\int_{[0,X_i]}\frac{\exp(\bm{\beta}^\top\bm{Z}_i)}{E^{(0)}(s;p,\bm \beta)}N_{n,0}(ds) \right)\right)\\ 
&-&\int_{[0,\tau]} \frac{E^{(0)}(t;1,\bm \beta)}{E^{(0)}(t;p,\bm \beta)}N_{n,0}(dt),
\end{eqnarray*}
and $D_{ki}^\tau =\mathbf{1}(X_i\leq \tau, A_i=k),$ $k\in\{0,2\}.$
The regression parameter $\beta$ is then estimated by
$$
\widehat \beta = \arg\max_{\beta\in B} \ell_n(\widehat p,\bm \beta;\tau),
$$
where $B\subset\mathbb{R}^q$ is a set of parameters and $\tau$ is fixed by the statistician. For  theoretical results, one needs conditions allowing to control for small values of $H_0([\tau,\infty))+ pH_1([\tau,\infty)).$ This is technical condition that is usually ignored in practice where one would simply take $\tau$ equal to the largest uncensored observation.
Next, the cumulative hazard function is estimated by
$$
\widehat \Lambda (t) = \Lambda_n (t;\widehat p, \widehat \beta)
$$
and the conditional  survival function of the lifetime of interest is estimated by
$$
\widehat S_T(t\mid \bm{z}) = \prod_{s\in(0,t]} \left(1 - \exp(\widehat \beta^\top\bm{z}) \widehat \Lambda (ds)\right)  
,\qquad t < \tau.
$$
The conditional cure probability $\mathbb{P}(T=\infty\mid \bm{Z}= \bm{z})$ is then estimated by
$$
\widehat S_T(\infty\mid \bm{z}) = \widehat S_T(\tau \mid \bm{z}) = \prod_{s\in(0,\tau]} \left(1 - \exp(\widehat \beta^\top\bm{z}) \widehat \Lambda (ds)\right) . 
$$

\subsection{The left-censoring and current status data case}

In the case of the model for left-censored and current status data
%
the estimate of $p_0$ is 
$$
\widehat p=\frac{\sum_{i=1}^n \textbf{1}(A_i=0)}{\sum_{i=1}^n \textbf{1} (A_i\neq 2)}.
$$
Next, using  the same notation as above,  let us define 
$$
F_{n,k}^{(l)}(t;\bm{\beta})=\frac{1}{n}\sum_{i=1}^n \exp(\bm{\beta}^\top\bm{Z}_i)\bm{Z}_i^{\otimes l}\mathbf{1}(X_i\leq t,A_i=k),\qquad l=0,1, \quad k\in\{0,1,2\}.
$$
Let us denote
$$
L^{(l)}_n(t;p, \beta)=F_{n,0}^{(l)} (t;\bm \beta)+p F_{n,2}^{(l)} (t;\bm \beta)
$$
and, for any $t$ such that  $L^{(0)}(t;p, \beta)>0$, consider
$$
R_n(dt;p,\beta )=
\frac{N_{n,0}(dt)}{L^{(0)}_n(t;p,\bm \beta)}. 
$$
Let us fix some (small) value $\varrho$ such that $H_0([0,\varrho])+H_2([0,\varrho]) >0$ and, similarly to the construction presented in section \ref{sec:profile_right}, define the approximated log-likelihood function
\begin{eqnarray*}
\ell_n(p,\bm \beta; \varrho)
&=&\frac{1}{n}\sum_{i=1}^n  D_{0i}^\varrho \left(\bm\beta^\top\bm Z_i-\log\left(L^{(0)}_n(X_i;p,\bm \beta)\right)\right)\\
&+& \frac{1}{n}\sum_{i=1}^n D_{1i}^\varrho \log\left(1-\exp\left( -\int_{[X_i,\infty)}\frac{\exp(\bm{\beta}^\top\bm{Z}_i)}{L^{(0)}_n(s;p,\bm \beta)}N_{n,0}(ds) \right)\right)\\
&-&\int_{[\varrho,\infty)}\frac{L^{(0)}_n(t;1,\bm \beta)}{L^{(0)}_n(t;p,\bm \beta)}N_{n,0}(dt),
\end{eqnarray*}
where $D_{ki}^\varrho =\mathbf{1}(X_i\geq \varrho, A_i=k),$ $k\in\{0,1\}.$
The regression parameter $\beta$ is then estimated by
$$
\widehat \beta = \arg\max_{\beta\in B} \ell_n(\widehat p,\bm \beta;\varrho),
$$
where $B\subset\mathbb{R}^q$ is a set of parameters and $\varrho$ is fixed by the statistician.  Like in the previous model, imposing a bound $\varrho$, here it should be a lower one, is a technical condition usually ignored in applications.
Next, the conditional  distribution function of the lifetime of interest is estimated by
$$
\widehat F_T(t\mid \bm{z}) =  \prod_{(t,\infty)}  \left(1 - \exp(\widehat \beta^\top\bm{z})  R_ n (ds;\widehat p, \widehat \beta )\right),\qquad t \geq  \varrho.
$$
The zero lifetime conditional probability $\mathbb{P}(T= 0 \mid \bm{Z}= \bm{z})$ is then estimated by
$$
\widehat F_T(0\mid \bm{z}) = \widehat F_T(\varrho \mid \bm{z}) 
$$
and the baseline cumulative reverse hazard is estimated by $\widehat R(t)=\int_{(t,\infty)} R_ n (ds;\widehat p, \widehat \beta)$.

\section{Asymptotic results}\label{sec:as_res}
\setcounter{equation}{0}

For the asymptotic results we only consider the investigation of the right-censored and current status data case. For the left-censored  and current status data case the results are similar and could be obtained after obvious modifications. 

Let $P$ be the probability distribution of $(X,A,\bm Z)$ and for any integrable function $f$ let $Pf=\mathbb{E}[f(X,A,\bm Z)].$ Let
$$
\mathbb{P}_n = \frac{1}{n} \sum_{i=1}^n \delta_{(X_i,A_i,\bm Z _i)}
$$
be the empirical distribution function and 
$
\mathbb{G}_n  = \sqrt{n} \left( \mathbb{P}_n - P\right). 
$

\medskip

Let us introduce the following additional assumptions.

\smallskip

\begin{itemize}
\item[\textbf{A2:}] The vector of covariates $\bm Z$ lies in $\mathbb{R}^{q}$, with $q\geq 1$ fixed, has a positive definite variance and is bounded, that is $\|\bm Z\|\leq c$ a.s. Moreover, $\bm\beta_0$ is an interior point of the parameter set $B$ that is a compact subset of $\bR^q$, and $p_0\in [\epsilon, 1 - \epsilon] \subset (0,1)$;

\item[\textbf{A3:}] The value  $\tau>0$ is such that
$
 H_0([\tau,\infty)) + H_1([\tau,\infty)) >0.
$
\end{itemize}

For simplicity we rule out the case $p_0=1$ because in this case $\mathbb{P}(A=2)=0$ and  $\widehat p = 1$ a.s., that is we are exactly in the classical PH model under right-censoring. Since $p_0$ is strictly positive, Assumption \textbf{A3}) is equivalent to $H_0([\tau,\infty)) + p_0 H_1([\tau,\infty)) >0.$ Also for simplicity, in the sequel we assume that the lifetime of interest $T$ and the censoring time $C$ are almost surely different. Let us notice that the construction we propose in sections \ref{sec:right-c} and \ref{sec:left-c} adapts to the case where $q$ depends on the sample size, or to the case where $\bm{\mathcal{Z}}$ is an infinite-dimensional space. The study of the properties of the estimators defined in such cases is left for future work.

\medskip

\begin{theor}[Consistency]\label{consist_lem} Let $\widehat \theta = (\widehat p,\widehat \beta^\top  )^\top $. Assume $\mathbb P (T=C)=0$ and Assumptions \textbf{A1}--\textbf{A3} hold true. Then:
\begin{enumerate}
\item $\widehat \theta \to \theta _0$, in probability;
\item $\sup_{t\in[0,\tau]}\left|\widehat{\Lambda}(t)-\Lambda_0(t)\right|\to 0$ in probability.
\end{enumerate}\label{th:consistency}
\end{theor}

\medskip

\begin{theor}[I.i.d. representation]\label{iid_rep} Under the assumptions of Theorem \ref{th:consistency} we have:
$$
\sqrt{n}\left(\begin{array}{c} \widehat p-p_0\\\widehat{\bm \beta}-\bm\beta_0\\ \widehat{\Lambda}(t)-\Lambda_0(t)\end{array}\right)= \bG_n\tilde\ell_{t;p_0,\bm\beta_0,\Lambda_0}+R_n(t), \qquad t\in[0,\tau],
$$
where
$
\ell_{s;p_0,\bm\beta_0,\Lambda_0}
$
is some squares integrable function 
and $R_n(t)$ is a reminder term  that is uniformly negligible, that is $\sup_{t\in[0,\tau]} |R_n(t)| \!= o_{\mathbb{P}}(1).$
\label{th_iid}
\end{theor}

\medskip

\begin{cor}[CLT] \label{cor_clt}
Under the assumptions of Theorem \ref{th_iid}
$$
\sqrt{n}\left(\begin{array}{c} \widehat p-p_0\\ \widehat{\bm \beta}-\bm\beta_0\\ \widehat{\Lambda}(\cdot)-\Lambda_0(\cdot)\end{array}\right) \leadsto {\cal G}\; \mbox{ in } \; {\bR}^{q+1}\times \ell^\infty([0,\tau]),
$$
where $\cal G$ is a tight, zero-mean Gaussian process with covariance function $$\rho_{\cal G}(s,t)=P\tilde\ell_{s;p_0,\bm\beta_0,\Lambda_0}\tilde\ell_{t;p_0,\bm\beta_0,\Lambda_0}^{\;\top},\qquad 0\leq s,t\leq \tau.$$
\end{cor}

\medskip

We could also derive the asymptotic law of the estimator of the survivor function $ S_T(t\mid \bm{z})$ for an arbitrary value $\bm{z}$ in the support of the covariates. The following result is a straightforward  extension of classical results for Cox PH model, see Link (1984).

\medskip

\begin{cor}[CLT for the conditional survivor] \label{cor_clt2}
Under the assumptions of Theorem \ref{th_iid} and for any fixed $\bm z \in\mathcal{Z},$
$$
\sqrt{n}\left(\widehat S_T(\cdot \mid \bm{z}) - S_T(\cdot \mid \bm{z}) \right) \leadsto {\cal S}_{\bm z}\; \mbox{ in } \; \ell^\infty([0,\tau]),
$$
where $\cal S_{\bm z}$ is a  tight, zero-mean Gaussian process.
\end{cor}

\medskip

Let us now investigate the estimator of the cure rate. Suppose that $H_0(\cdot)$ has a bounded support and let $\tau_{H_0}$ be its right endpoint. Assume that  $H_1([\tau_{H_0},\infty))>0.$ Then in our model we necessarily have $\Lambda ([0,\tau_{H_0}])<\infty$ and
$\inf_{\|\bm z\| \leq c } S_T(\tau_{H_0}\mid \bm z) >0.$ Since one cannot identify the law of $T$ beyond the last uncensored observation, by an usual convention, $S_T(\infty\mid \bm z)=S_T(\tau_{H_0}\mid \bm z).$ These quantities could be estimated by $\widehat S_T(\tau\mid \bm{z}).$ The following corollary is a direct consequence of Corollary \ref{cor_clt2}.  

\medskip

\begin{cor}[CLT for the conditional cure rate] \label{cor_clt3}
Suppose that the assumptions \textbf{A1}, \textbf{A2} hold true. Moreover,
$H_0(\cdot)$ has a bounded support with right endpoint  $\tau_{H_0}<\infty$. Assume that  $H_1([\tau_{H_0},\infty))>0.$ Then
$$
\sqrt{n}\left(\widehat S_T(X_{(n)}^0 \mid \bm{z}) - S_T(\infty\mid \bm{z})\right) \leadsto N(0,V(\bm z)),
$$
where $X_{(n)}^0$ is the largest uncensored observation and
$
V(\bm z) = \mathbb E ({\cal S}_{\bm z}(\tau_{H_0}))
$
with $\cal S_{\bm z}$ from Corollary \ref{cor_clt2}. 
\end{cor}

\medskip

The estimation of the covariance functions 
of the processes $\mathcal G$ and $\mathcal S_{\bm z}$,
and of the variance $V(\bm z)$ is quite difficult. Therefore we propose an alternative route, based on the weighted bootstrap, for estimating the asymptotic law of our estimators. Let us consider  $\tilde\ell_{\cdot; \widehat p,\widehat{\bm\beta} ,\widehat \Lambda}$ that is  an uniformly consistent estimator of $\tilde\ell_{\cdot;p_0,\bm\beta_0,\Lambda_0}$. Next, let us define
$$
\bG_n'= \frac{1}{\sqrt{n}}\sum_{i=1}^n(\xi_i-\bar\xi)\delta_{(X_i,A_i,\bm Z_i)}
$$
where $\xi_1,\dots,\xi_n$ are i.i.d., with zero mean and unit variance random variables, for instance gaussian, independent of the data.

\medskip

\begin{theor}[Asymptotic law approximation] Under the assumptions of Theorem \ref{th_iid}:
$$
\left(\bG_n\tilde\ell_{\cdot;p_0,\bm\beta_0,\Lambda_0},\bG_n'\tilde\ell_{\cdot;\widehat p,\widehat{\bm\beta},\widehat \Lambda}\right)\leadsto ({\cal G},{\cal G'})\mbox{ in }  \left({\bR}^{q+1}\times \ell^\infty([0,\tau])\right)^2
$$
where ${\cal G}$ and ${\cal G}'$ are independent and identically distributed.
\label{var_est}
\end{theor}

\medskip

As a direct consequence of Theorem \ref{var_est} one could obtain the validity of the bootstrap approximation of the asymptotic laws stated in Corollaries \ref{cor_clt} to \ref{cor_clt3}. The details are omitted.

\qquad

\section{Appendix}
\setcounter{equation}{0}

\subsection{Notation}


For any matrix $A$, we denote by $\|A\| = \sqrt{Trace (A^\top A)} .$ Let us recall that vectors are considered as column matrices. The spaces of functions we consider are endowed with the uniform (supremum) norm that is denoted by $\|\cdot\|_\infty.$
Let $\partial_p$ and $\partial_\beta$ denote the partial derivation operators with respect to $p$ and $\beta,$ respectively.

Let
$$
s_{k}^{(l)}(t;\bm{\beta}) = \mathbb E \left\{ S_{n,k}^{(l)}(t;\bm{\beta}) \right\} = \!\mathbb E \left\{ \exp(\bm{\beta}^\top\bm{Z})\bm{Z}^{\otimes l}\mathbf{1}(X\geq t,A=k) \right\},$$
and
$$
e^{(l)}(t;\bm \theta) = e^{(l)}(t;p, \bm \beta) = \mathbb E \left\{ E^{(l)}_n(t;\bm \theta) \right\} = s_{0}^{(l)}(t;\bm{\beta}) + p s_{1}^{(l)}(t;\bm{\beta}), \qquad l=0,1, \; k\in\{0,1,2\}.
$$
 Let 
\begin{multline*}
\ell(p,\bm \beta;\tau) = \mathbb{E} [\bm\beta^\top\bm Z \; \mathbf{1}(X \leq \tau, A=0) ]- \int_{[0,\tau]}  \log\left(e^{(0)}(t;p,\bm \beta)\right)H_0(dt)\\
+ \mathbb E \left[   \log\left(1-\exp\left( -\exp(\bm{\beta}^\top\bm{Z}) \Lambda(X;p,\bm \beta) \right)\right)  \mathbf{1}(X \leq \tau, A=2)\right]\\
-\int_{[0,\tau]} \frac{e^{(0)}(t;1,\bm \beta)}{e^{(0)}(t;p,\bm \beta)}H_{0}(dt).
\end{multline*}
The criterion $\ell(p,\bm \beta;\tau)$ is expected to be the limit of the approximated log-likelihood function $\ell_n(p,\bm \beta;\tau)$. 
Let us recall that  $P$ denotes the probability distribution of $(X,A,\bm Z)$ and for any integrable function $f$ let $Pf=\mathbb{E}[f(X,A,\bm Z)].$ Moreover,
$$
\mathbb{P}_n = \frac{1}{n} \sum_{i=1}^n \delta_{(X_i,A_i,\bm Z _i)}
$$
is the empirical measure, and $
\mathbb{G}_n  = \sqrt{n} \left( \mathbb{P}_n - P\right). 
$ 
Finally, define
$$
\delta_k(a) = \mathbf 1 (a=k),\qquad k\in \{0,1,2\}.
$$

\subsection{Proof of Theorem \ref{consist_lem}}

To prove consistency for $\widehat \beta$, it suffices, for instance, to use the results from section 5.2 of  van der Vaart (1998). This means to check that
\begin{equation}\label{idef}
\ell(p_0,\bm \beta_0 ;\tau)  > \ell(p,\bm \beta;\tau) ,\qquad \forall (p,\bm \beta^\top)^\top \in [\epsilon,1-\epsilon] \times B,\; (p,\bm \beta^\top)^\top \neq (p_0,\bm \beta_0^\top )^\top ,
\end{equation}
\begin{equation}\label{uni1}
\sup_{p\in[\epsilon, 1-\epsilon]} 
\sup_{\beta \in B} \left| \ell_n (p,\bm \beta ;\tau)  - \ell(p,\bm \beta;\tau) \right|  = o_{\mathbb P}(1),
\end{equation}
and the map $ (p,\bm \beta^\top)^\top \mapsto  \ell(p,\bm \beta;\tau)$ is continuous. The continuity condition is a direct consequence of 
the Lebesgue's Dominated Convergence Theorem. Conditions \eqref{idef} and \eqref{uni1} will be consequence of the two following lemmas. 

\begin{lem}\label{lem_idef}
Under the conditions of Theorem \ref{consist_lem}, the true value of the parameter $\theta = (p,\beta^\top)^\top$ is identifiable, that is the condition \eqref{idef} holds. 
\end{lem}

\begin{proof}[Proof of Lemma \ref{lem_idef}]
Consider the conditional log-likelihood of the multinomial variable $A\in\{0,1,2\}$ given $\bm Z = \bm z$
\begin{multline*}
\log p (t,A,\bm z; p,\beta) = \delta_0(A)    
\exp(\bm{\beta}^\top\bm{z}) \left\{H_0([t,\infty)|\bm{z})+pH_1([t,\infty)|\bm{z}) \right\}\Lambda(dt)\\
+ \delta_1(A)  \exp\left(-\exp(\bm\beta^\top\bm z) \Lambda([0,t])\right)
F_C(dt|\bm z)\\
+\delta_2(A)    \left\{1-  \exp\left(-\exp(\bm\beta^\top\bm z) \Lambda([0,t])\right)  \right\}F_C(dt|\bm z).
\end{multline*}
Following the notation of Gill (1994), here we treat dt not just as the length
of a small interval $[t, t + dt)$ but also as the name of the interval itself.
Note that 
$$
\log p (t,A,\bm z; p_0,\beta_0) =  \delta_0(A)  H_0(dt|\bm{z}) + \delta_1(A)   H_1(dt|\bm{z}) + \delta_2(A)   H_2(dt|\bm{z}).
$$
By the standard log-likelihood ratio inequality, for any $t$ and $\bm z$,  we have
$$
\mathbb E \left[ \log \frac{p (t,A,\bm z; p,\beta)}{p (t,A,\bm z; p_0,\beta_0) } \bigg| X \in dt, \bm Z = \bm z \right] \leq 0.
$$
Integrating with respect to $X\in [0,\tau]$ and $\bm Z$, we obtain 
$$
\mathbb E \left[ \log \frac{p (X,A,\bm Z; p,\beta)}{p (X,A,\bm Z; p_0,\beta_0) }   \mathbf 1 (X\in [0,\tau]) \right] \leq 0.
$$
If the last inequality becomes equality, then necessarily 
$p (t,0,\bm z; p,\beta)=p (t,0,\bm z; p_0,\beta_0) $ for almost all $t\in[0,\tau]$ in the support of $X$ and $\bm z$ in $\mathcal Z$. With our assumptions, this cannot happen when  $ (p,\bm \beta^\top)^\top \neq (p_0,\bm \beta_0^\top )^\top $.  
\end{proof}

\quad

\begin{lem}\label{lem_uni1}
Under the conditions of Theorem \ref{consist_lem}, the  condition \eqref{uni1} holds. Moreover, 
 $$
\sup_{p\in[\epsilon, 1-\epsilon]} 
\sup_{\beta \in B}\sup_{t\in[0, \tau]}  \left| \Lambda_n (t;p,\bm \beta )- \Lambda (t;p,\bm \beta ) \right|  = o_{\mathbb P}(1).
 $$
\end{lem}

\begin{proof}[Proof of Lemma \ref{lem_uni1}]
First, note that 
\begin{equation}\label{conv_E}
\sup_{p\in[\epsilon, 1-\epsilon]} 
\sup_{\beta \in B}\sup_{t\in[0, \tau]}  \left| E^{(l)}(s;p,\bm \beta) - e^{(l)}(s;p,\bm \beta) \right|  = o_{\mathbb P}(1), \qquad l=0,1. 
\end{equation}
This is a consequence of the uniform law of large numbers for the classes of functions
$$
\{ (x,a,\bm z) \mapsto  \exp(\beta^\top \bm z) z^{\otimes l} \mathbf 1 (x\geq t) \delta_k(a): \beta\in B, t\in[0,\tau],k\in\{0,1,2\} \},\quad l=0,1.
$$
These two classes of functions are bounded and have polynomial complexity, that is they are VC classes. In particular, they are Glivenlo-Cantelli classes. Next, by our assumptions, for any, $p$, $\bm z$ and $l$, we have $t\mapsto  e^{(l)}(t;p,\bm \beta)$ is decreasing. Moreover, 
\begin{equation}\label{low_b}
\inf_{l\in\{0,1\}}\inf_{p\in[\epsilon, 1-\epsilon]} 
\inf_{\beta \in B}e^{(l)}(\tau;p,\bm \beta)> 0
\end{equation}
Recall,
$$
\Lambda (t;p,\bm \beta ) = \int_{[0,t]}\frac{H_{0}(ds)}{e^{(0)}(s;p,\bm \beta)}.
$$
Finally, we can write 
\begin{multline*}
\Lambda_n (t;p,\bm \beta )- \Lambda (t;p,\bm \beta )  = \int_{[0,t]} \frac{\delta_0(a)}{E^{(0)}(s;p,\bm \beta)} d \mathbb P_n(s,a,\bm z) - \int_{[0,t]}  \frac{\delta_0(a)}{e^{(0)}(s;p,\bm \beta)}
dP(s,a,\bm z) \\
= \int_{[0,t]} \left[ \frac{\delta_0(a)}{E^{(0)}(s;p,\bm \beta)}- \frac{\delta_0(a)}{e^{(0)}(s;p,\bm \beta)}\right] d \mathbb P_n(s,a,\bm z)  
+ \int_{[0,t]}  \frac{\delta_0(a)}{e^{(0)}(s;p,\bm \beta)}d(\mathbb P_n - P)(s,a,\bm z) 
\end{multline*}
and the result follows from \eqref{conv_E}, \eqref{low_b} and again the uniform law of large numbers. \end{proof}

\quad

To justify Theorem \ref{consist_lem}, it suffices to notice that Lemma \ref{lem_uni1} and the uniform law of large numbers guarantee condition \eqref{uni1} and to use Theorem 5.7 from  van der Vaart (1998). 

\subsection{Asymptotic normality}

In this section we sketch the arguments allowing to prove Theorem \ref{th_iid} and  Corollaries \ref{cor_clt} to \ref{cor_clt3}.

Note that 
$$
 \partial_\beta \Lambda(t;\theta_0) =  - \int_{[0,t]} \frac{e^{(1)}(s;\theta_0)}{[e^{(0)}(s;\theta_0)]^2} H_0(ds)  =- \int_{[0,t]} \frac{e^{(1)}(s;\theta_0)}{e^{(0)}(s;\theta_0)}\Lambda(ds;\theta_0) , \quad t\in[0,\tau]
,$$
with $\Lambda(\cdot;\theta_0)  = \Lambda(\cdot; p_0,\bm{\beta}_0) $ defined in \eqref{cum_haz}.
Next, define
$$
\wp_n (t;\theta) = \frac{\partial \Lambda_n (t;\theta)}{\partial p}= \frac{\partial \Lambda_n (t;p,\beta)}{\partial p} \qquad \text{ and } \qquad \wp_0 (t) = \frac{\partial \Lambda (t,\theta_0)}{\partial p} = \frac{\partial \Lambda (t;p_0,\beta_0)}{\partial p} .
$$

 Consider the score function
\begin{eqnarray*}
U_n(\theta;\tau)&=&U_n(p,\beta;\tau) \\
&=&\partial_\beta \ell_n(p,\beta;\tau)\\
&=&
\frac{1}{n}\sum_{i=1}^n \mathbf 1 (X_i\in [0,\tau]) \delta_0(A_i)\left(\bm Z_i- \frac{E^{(1)}(X_i;p,\beta)}{E^{(0)}(X_i;p,\beta)}\right)\\
&&-\frac{1}{n}\sum_{i=1}^n \mathbf 1 (X_i\in [0,\tau]) \delta_0(A_i) \\ &&\qquad \qquad \times \frac{E^{(1)}(X_i;1,\beta)E^{(0)}(X_i;p,\beta)-
E^{(1)}(X_i;p,\beta)E^{(0)}(X_i;1,\beta)}{\left[E^{(0)}(X_i;p,\beta)\right]^2}\\
&&+ \frac{1}{n}\sum_{i=1}^n \mathbf 1 (X_i\in [0,\tau]) \delta_2(A_i)
\frac{\exp\left(-V_i(\theta)\right)}{1-\exp\left(-V_i(\theta)\right)} W_i(\theta),
\end{eqnarray*}
where
$$
V_i(\theta)=\exp(\beta^\top \bm Z_i)\sum_{j=1}^n\frac{ \mathbf{1}(X_j\leq X_i ) }{E^{(0)}(X_j;\beta,p)}\delta_0(A_j)
$$
and
$$
W_i(\theta) =\exp(\beta^\top  \bm Z_i)\sum_{j=1}^n \mathbf{1}(X_j\leq X_i ) \delta_0(A_j)
\frac{\bm Z_i E^{(0)}(X_j;\beta, p)-E^{(1)}(X_j;\beta,
p)}{[E^{(0)}(X_j;\beta, p)]^2}.
$$
Since we imposed $\mathbb P (T=C)=0$, we could equivalently define $V_i(\theta)$ with $ \mathbf{1}(X_j\leq X_i )$ instead of $ \mathbf{1}(X_j < X_i )$, as it would require the definition of the approximate  log-likelihood $\ell_n$. Let us also consider $$U_0(\theta;\tau) = \partial_\beta \ell(p,\beta;\tau),$$ the limit of this score function. The following lemma is a simple consequence of the uniform law of large numbers and the convergence in probability of $U-$statistics, and hence the proof is omitted. 

\begin{lem}\label{prop1}
Under the Assumptions \textbf{A1}--\textbf{A3} and if 
 $\theta_n=(p_n,\bm \beta_n^\top)^\top \rightarrow \theta_0=(p_0,\bm \beta_0^\top )^\top$
in probability, then
\begin{enumerate}
\item $\left\|  \partial_\beta U_n(\theta_n ,\tau) - \partial_\beta U_0 (\theta_0 ,\tau)  \right\| = o_{\mathbb P}(1); $
\item $\left\| \partial_p U_n(\theta_n ,\tau) - \partial_p U_0 (\theta_0 ,\tau)  \right\| = o_{\mathbb P}(1). $
\item $\sup_{t\in [0,\tau]}\| \partial_\beta \Lambda_n(t;\theta_n) - \partial_\beta \Lambda(t;\theta_0) \| = o_{\mathbb P}(1); $
\item $\sup_{t\in [0,\tau]}|\wp_n (t;\theta_n) - \wp_0 (t)  | = o_{\mathbb P}(1). $
\end{enumerate}
\end{lem}

Let us sketch the arguments of the proof of Theorem \ref{iid_rep}. By the definition of $\widehat \beta$ and the first-order Taylor expansion of $U_n(\theta;\tau)$ in a neighborhood of  $\theta_0$,
$$
\sqrt{n} U_n(\widehat \theta ,\tau) = 0 = \sqrt{n} U_n(\theta_0 ,\tau) + \partial_\beta U_n(\theta^*_n ,\tau) \sqrt{n}(\widehat \beta - \beta_0) + \partial_p U_n(\theta^*_n ,\tau) \sqrt{n}(\widehat p - p_0),
$$
where $\theta_n^*$ is a point between $\widehat \theta$ and $\theta_0$. By Lemma \ref{prop1}, if the $q\times q-$matrix $\partial_\beta U_0 (\theta_0 ,\tau)$ is invertible,
$$
\sqrt{n}(\widehat \beta - \beta_0) = - \partial_\beta U_0 (\theta_0 ,\tau)^{-1}  \sqrt{n} U_n(\theta_0 ,\tau) - 
\partial_\beta U_0 (\theta_0 ,\tau)^{-1} \partial_p U_0(\theta_0 ,\tau)  \sqrt{n}(\widehat p - p_0) +o_{\mathbb{P}}(1).
$$ 
Hence, the asymptotic normality of $\sqrt{n}(\widehat \beta - \beta_0)$ will follow from the joint asymptotic normality of $ \sqrt{n} U_n(\theta_0 ,\tau) $ and $\sqrt{n}(\widehat p - p_0) $. On the other hand, by a Taylor expansion and Proposition \ref{prop1}, for some $\theta_n^\dag$  between $\widehat \theta$ and $\theta_0$, we can write
\begin{eqnarray*}
 \sqrt{n} \left(\Lambda_n(t;\widehat \theta) -  \Lambda(t;\theta_0)\right) &=&  \sqrt{n} \left(\Lambda_n(t;\theta_0) -  \Lambda(t;\theta_0)\right)\\
 &&+ \partial_\beta \Lambda_n(t;\theta_n^\dag)^\top \sqrt{n}(\widehat\beta - \beta_0) + \wp_n(t,\theta_n^\dag) \sqrt{n}(\widehat p - p_0)\\
 &=& \sqrt{n} \left(\Lambda_n(t;\theta_0) -  \Lambda(t;\theta_0)\right)\\
 &&+ \partial_\beta \Lambda(t;\theta_0)^\top \sqrt{n}(\widehat \beta - \beta_0) + \wp(t,\theta_0) \sqrt{n}(\widehat p - p_0)+  o_{\mathbb{P}}(1).
\end{eqnarray*}
Hence, the asymptotic normality of $\sqrt{n} \left(\widehat \Lambda(t) -  \Lambda_0(t)\right) $ will  follow from the joint asymptotic normality of $ \sqrt{n} \left(\Lambda_n(t;\theta_0) -  \Lambda(t;\theta_0)\right),$ $ \sqrt{n} U_n(\theta_0 ,\tau) $ and $\sqrt{n}(\widehat p - p_0) $. Gathering facts, we have
$$
\sqrt{n} 
\begin{pmatrix}
\widehat p - p_0\\
\widehat \beta - \beta_0 \\
\widehat \Lambda(t) -  \Lambda_0(t)
\end{pmatrix}
= \Sigma_1 (t) \sqrt{n}
\begin{pmatrix}
U_n(\theta_0 ,\tau) \\
\Lambda_n(t;\theta_0) -  \Lambda(t;\theta_0) \\
\widehat p - p_0
\end{pmatrix}
+ o_{\mathbb{P}}(1),
$$
where
$$
\Sigma_1 (t) \! = \! 
  \begin{pmatrix}
  0 & 0 & 1\\
  - \partial_\beta U_0 (\theta_0 ,\tau)^{-1}  &  0  &  -  \partial_\beta U_0 (\theta_0 ,\tau)^{-1} \partial_p U_0(\theta_0 ,\tau)\\
  \!\! - \partial_\beta \Lambda(t;\theta_0)^\top \partial_\beta U_0 (\theta_0 ,\tau)^{-1}   &  1  &   \wp(t,\theta_0)  -   \partial_\beta \Lambda(t;\theta_0)^\top \partial_\beta U_0 (\theta_0 ,\tau)^{-1} \partial_p U_0(\theta_0 ,\tau) 
 \end{pmatrix} .
$$
Hence, it suffices to study the asymptotic behavior of the $(q+2)-$dimension vector  
$$
\sqrt{n}\left(  
U_n(\theta_0 ,\tau) ^\top,
\Lambda_n(t;\theta_0) -  \Lambda(t;\theta_0) ,
\widehat p - p_0
\right)^\top.
$$

\subsubsection{I.i.d. representation of $ \widehat p $}

It is clear that the class of 0/1-valued functions $\delta_k(\cdot)$ defined on $\{0,1,2\}$ and indexed by $k\in\{0,1,2\}$ is $P-$Donsker. 
We have
$$
\widehat p = \frac{\mathbb{P}_n  \delta_0}{\mathbb{P}_n ( \delta_0 + \delta_2)}.
$$
Using the first-order Taylor expansion for  $f(x_1,x_2) = x_1/(x_1+x_2)$ with $x_1,x_2\geq c$ for $c$ some small positive constant, we easily derive the  representation
\begin{eqnarray*}
\sqrt{n} (\widehat p - p) &=& \sqrt{n}\left[ f( \mathbb{P}_n  \delta_0, \mathbb{P}_n  \delta_2) - f(P  \delta_0, P  \delta_2)   \right]\\
&=& \frac{\partial f}{\partial x_1}(P \delta_0, P  \delta_2) \mathbb{G}_n \delta_0 +  \frac{\partial f}{\partial x_2}(P \delta_0, P  \delta_2) \mathbb{G}_n \delta_2 + o_{\mathbb{P}}(1)\\
&=& \frac{P\delta_2}{(P \delta_0+ P  \delta_2)^2}\mathbb{G}_n \delta_0 - \frac{P\delta_0}{(P \delta_0+ P  \delta_2)^2}\mathbb{G}_n \delta_2 + o_{\mathbb{P}}(1).
\end{eqnarray*}

\subsubsection{I.i.d. representation of $\Lambda_n(t;\theta_0) $}

For any $t\geq 0,$ 
$k,l\in\{0,1,2\}$, let us define 
$$
(x,a,z) \!\mapsto f_{t}^{(k,l)} (x,a,{\bm z}) =  \exp(\beta_0^\top \!{\bm z}) {\bm z}^{\otimes l}\mathbf{1}(x\geq t)\delta_k(a) ,\quad x\geq 0, a\in\{0,1,2\},{\bm z} \in\mathcal{Z}\!\subset\! \mathbb{R}^q.
$$
Thus $Pf_{t}^{(k,l)}= s_{k}^{(l)}(t;\beta_0)$. 
For any $k,l\in\{0,1,2\}$  consider the family of such functions
$$
\mathcal{F}^{(k,l)} = \left\{ f_{t}^{(k,l)} (\cdot,\cdot,\cdot) : t\in[0,\tau]\right\}.
$$
Each of such families are clearly uniformly bounded and $P-$Donsker. Next, let
$$
e_t^{(0)}(x,a,{\bm z}) = f_{t}^{(0,0)} (x,a,{\bm z}) + p_0 f_{t}^{(1,0)} (x,a,{\bm z}).
$$
Thus $Pe_t^{(0)} = e^{(0)} (t;\theta_0)$ and we can rewrite 
$$
\Lambda_n(t;\theta_0) = \mathbb{P}_n\left[ \frac{\delta_0(a)\mathbf{1}(x\leq t)}{\mathbb{P}_n e_x^{(0)} } \right]
$$
and
$$
\Lambda(t;\theta_0) = P\left[ \frac{\delta_0(a)\mathbf{1}(x\leq t)}{P e_x^{(0)} } \right].
$$
Hence, we can write
$$
\sqrt{n} \left( \Lambda_n(t;\theta_0)-\Lambda(t;\theta_0) \right) =\sqrt{n}\left[ \int_{[0,t]} \frac{1}{A_n} dB_n - \int_{[0,t]} \frac{1}{A} dB \right]
$$
for the c\`adl\`ag functions
$$
A (t)= P e_t^{(0)}, \qquad B(t) = P\left[ \mathbf{1}(x\leq t) \delta_0(a)\right] = p_0 \int_{[0,t]} A(s) \Lambda(ds;\theta_0),\qquad t\in[0,\tau],
$$
and 
$
A_n, B_n
$
their empirical version obtained by replacing $P$ by $\mathbb{P}_n.$

Let $D[a,b]$ be the space of c\`adl\`ag functions on $[a, b]$ and let 
$BV_M[a, b]$ be the set of all
functions $B \in D[a, b]$ with total variation $|B(0)|+\int_{(a,b]}
|B(ds)| \leq M.$ Let $$\mathbb{D}_M= \{A \in D[0,\tau]: A\geq \epsilon\} \times BV_M[0, \tau] $$  for some  positive constants $\epsilon, M$. For sufficiently small $\epsilon$ and sufficiently large $M$ (depending on $c$ from assumption \textbf{A2} and $ H_0([\tau,\infty)) + H_1([\tau,\infty)) >0$ from assumption \textbf{A3}, $A,$  $B$  and, with probability tending to 1, $A_n$, $B_n$ defined above belong to $\mathbb{D}_M$. The $ D[0, \tau]-$valued map $(A,B)\mapsto \int_{[0,\cdot]}(1/A)dB$ is Hadamard differentiable on the set $\mathbb{D}_M$ and the derivative map is  given by 
$$
(\alpha,\beta) \mapsto \int_{[0,\cdot]} (1/A)d\beta - \int_{[0,\cdot]}(\alpha/A^2) dB;
$$
see, for instance, Kosorok (2008)  section 12.2.  The integral $ \int_{[0,\cdot]} (1/A)d\beta$ is defined via integration by parts if $\beta$ is not of bounded variation. To derive the i.i.d. representation, let us use  the Hadamard derivative with 
$$
\alpha  = \mathbb{G}_n e_\cdot^{(0)},\quad \beta =  \mathbb{G}_n \left[ \delta_0(a)\mathbf{1}(x\leq \cdot) \right].
$$
Since $e^{(0)}(t;\theta_0) = Pe_s^{(0)}$.  Deduce that for any $t\in[0,\tau]$,
\begin{eqnarray*}
\sqrt{n}\left( \Lambda_n(t;\theta_0)-\Lambda(t;\theta_0) \right) &=&  \mathbb{G}_n f_t^{(2)}\\
&&- p_0 \int_{[0,t]} \left\{ \mathbb{G}_n f^{(0,0)}_s    \right\} \frac{\Lambda (ds;\theta_0)}{ e^{(0)}(s;\theta_0)}\\&& -  p_0^2  \int_{[0,t]}\left\{\mathbb{G}_n  f^{(1,0)}_s   \right\}\frac{\Lambda (ds;\theta_0)}{ e^{(0)}(s;\theta_0)}
\\&&+ r_n(t),
\end{eqnarray*}
with $\sup_{[0,\tau]}|r_n(t)| = o_{\mathbb{P}}(1),$ where
$$
f_t^{(2)} \in\mathcal{F}^{(2)} =\left\{ (x,a,{\bm z})\mapsto f_t^{(2)} (x,a,{\bm z}) =  \delta_0(a)\mathbf{1}(x\leq t)  \left[{ p_0 e^{(0)}(t;\theta_0) } \right]^{-1}  : t\in[0,\tau] \right\}.
$$
Clearly, $\mathcal{F}^{(2)}$ is a $P-$Donsker class of real-valued functions defined on $[0,\tau]\times \{0,1,2\}\times \bm{\mathcal{Z}}$.

\subsubsection{I.i.d. representation of $\sqrt{n}  
U_n(\theta_0 ,\tau) $}

We only present the guidelines that could be followed to deduce the asymptotic normality $\sqrt{n}  
U_n(\theta_0 ,\tau) $.
Consider the function $\phi :  \mathbb{R} \times \mathbb{R} \color{black} \times \mathbb{R}^q\times \mathbb{R}^q \mapsto \mathbb{R}^d$ given by the relationship
$$
\phi(y_1,y_2,y_3,y_4) = - \frac{y_3+p_0y_4}{y_1+p_0y_2}  - \frac{y_3 + y_4}{y_1+p_0y_2} +   \frac{y_1+y_2}{(y_1+p_0y_2)^2}(y_3+p_0y_4).
$$
Let $\Xi  $ be a set of $(3+3q)-$dimension vector valued functions of the observed variables
$$\eta = (\eta_1,\eta_2,\eta_3^\top,\eta_4^\top,\eta_5,\eta_6^\top):[0,\tau] \rightarrow  \mathbb{R}\times \mathbb{R}\times\mathbb{R}^q\times\mathbb{R}^q\times\mathbb{R}\times\mathbb{R}^q, $$ 
such that each component of  $\eta$ is a monotone c\`adl\`ag function bounded in absolute value by some 
sufficiently large constant $M$.  Moreover, we assume that the function 
$$
\eta_0 (x)= (Pf_x^{(0,0)}, Pf_x^{(1,0)}, (Pf_x^{(0,1)})^\top, (Pf_x^{(1,1)})^\top,\Lambda(x;\theta_0), (\partial_\beta \Lambda(x;\theta_0))^\top ),\qquad x\in[0,\tau],
$$
belongs to $\Xi$, and, with probability tending to 1 as $n\rightarrow \infty$, the empirical version 
$$
\eta_n (x) = (\mathbb{P}_nf_x^{(0,0)},  \mathbb{P}_n f_x^{(1,0)}, (\mathbb{P}_nf_x^{(0,1)})^\top, (\mathbb{P}_nf_x^{(1,1)})^\top,\Lambda_n(x;\theta_0), (\partial_\beta \Lambda_n (x;\theta_0))^\top ),\quad x\in[0, \tau],
$$
is also contained in $\Xi$. 
Let us define the family of functions 
$$
\mathcal{H} = \{ (x,a,z)\mapsto h_\eta (x,a,z): \eta \in \Xi \},
$$
where
\begin{multline*}
h_\eta (x,a,z) =\bigg[   \delta_0(a) \left\{ z + \phi(\eta_1(x),\eta_2(x), \eta_3(x),\eta_4(x))  \right\}   \\
\left. +\delta_2(a)\frac{\exp(-\exp(\beta_0^\top {\bm z})\eta_5(x))}{1-\exp(-\exp(\beta_0^\top {\bm z})\eta_5(x))} \{  z\eta_5(x) + \eta_6(x)\} \right]\mathbf{1}(x\leq \tau ).
\end{multline*}
Next, the idea is to decompose
$$
\sqrt{n}U_n (\theta_0, \tau) = \mathbb{G}_n h_{\eta_0} + \sqrt{n} P h_{\eta_n} + \mathbb{G}_n(h_{\eta_n}- h_{\eta_0}).
$$
By  the continuity of the paths of the empirical process, see for instance Theorem 2.1 of van der Vaart \& Wellner
(2007), 
$
 \mathbb{G}_n(h_{\eta_n}- h_{\eta_0}) = o_{\mathbb{P}}(1). 
$ 
The term  $ \mathbb{G}_n h_{\eta_0} $ is already under a convenient form and could be handled by standard CLT.

\end{document}